\documentclass[a4, 12pt]{amsart}
\usepackage{amsmath,amsthm,amsfonts,amssymb,amscd}
\newtheorem{theorem}{Theorem}[section]

\usepackage{xcolor}
\usepackage{mathrsfs}
\usepackage{blkarray}
\allowdisplaybreaks

\newtheorem{definition}[theorem]{Definition}

\newtheorem{prop}[theorem]{Proposition}
\newtheorem{corollary}[theorem]{Corollary}
\newtheorem{question}[theorem]{Question}
\newtheorem{remark}[theorem]{Remark}

\numberwithin{equation}{section}

\begin{document}
\title[$\mathcal{AM}$-Toeplitz and $\mathcal{AN}$-Hankel operators]{Absolutely minimum attaining Toeplitz and absolutely norm attaining Hankel operators}
\author{G. Ramesh }
\address{G. Ramesh, Department of Mathematics, IIT Hyderabad, Kandi, Sangareddy, Telangana- 502284, India.}
 \email{rameshg@math.iith.ac.in}
\author{Shanola S. Sequeira}
 \address{Shanola S. Sequeira, Department of Mathematics, IIT Hyderabad, Kandi, Sangareddy, Telangana- 502284, India.}
 \email{ma18resch11001@iith.ac.in}
\maketitle
\begin{abstract}
	In this article, we completely characterize absolutely norm attaining Hankel operators and absolutely minimum attaining Toeplitz operators. We also improve \cite[Theorem 2.1]{RGSSSTOE1}, by characterizing the absolutely norm attaining Toeplitz operator $T_\varphi$ in terms of the symbol $\varphi \in L^\infty$.
\end{abstract}

\section{Introduction}
Let $L^2$ denote the Lebesgue space of all square integrable functions on the unit circle $\mathbb{T}$ in the complex plane with respect to the normalized Lebesgue measure $\mu$ and $L^\infty$ be the Banach space of all complex valued essentially bounded $\mu$-measurable functions on $\mathbb{T}$. Let $H^2$ denote the closed subspace of $L^2$ consisting of all those functions whose negative Fourier co-efficients vanish.

For any $\varphi \in L^\infty$, the Toeplitz operator $T_\varphi :H^2 \to H^2$ is defined by
\[T_\varphi f = P (\varphi f) \quad  \text{for all} \ f  \in H^2,\]
and the Hankel operator  $H_\varphi :H^2 \to H^2$ is defined by
\[H_\varphi f = J(I-P)\varphi f \quad \text{for all} \ f  \in H^2,\]
where $(\varphi f)(z) = \varphi(z) f(z)$ for all $z \in \mathbb{T}$, $P$ is an orthogonal projection of $L^2$ onto $H^2$ and $J(z^{-n}) = z^{n-1}, n= 0,\pm1,\pm2, \dots$ is the unitary operator on $L^2$.

Toeplitz and Hankel operators are well studied in matrix theory, operator theory, function theory etc., as they find applications in many fields. The initial study of Toeplitz operators was done by Otto Toeplitz and thereafter Brown and Halmos in \cite{BROWN} discussed some of the important properties of them. Further, H.Widom, R. G. Douglas, D. Sarason etc., also worked on Toeplitz operators and most of their results can be found in \cite{DOU}. We refer to \cite{MAR,Peller} for the results on Hankel operators.

One of the important properties of any bounded linear operator is it's norm attaining property. Such a property is vastly studied on Banach spaces (for example, see \cite{AcostaBPBP,BishopPhelps,Lindenstrauss}). A few developments of norm attaining operators on Hilbert spaces can be found in \cite{Carvjalnorm,EnfloKoverSmithies}. For more details on norm attaining Toeplitz and Hankel operators we refer to \cite{Brownpariso,YOS2}. It is important to note that the set of all compact operators and isometries defined between the Hilbert spaces are always norm attaining. Moreover, on restriction of such operators to any closed subspace of the Hilbert space is also norm attaining. This observation initiated the study of new class of operators by Carvajal and Neves \cite{Carvjalnorm} and called them the absolutely norm attaining operators or $\mathcal{AN}$-operators.

A natural counterpart of the norm attaining and the absolutely norm attaining operators are the minimum attaining and the absolutely minimum attaining operators, which are also introduced by Carvajal and Neves in \cite{CARmin}.

Due to the lack of non-zero Toeplitz compact operators on $H^2$, the authors in \cite{RGSSSTOE1} were interested to look at $\mathcal{AN}$-Toeplitz operators, since this set is non-empty as it contains isometric Toeplitz operators. The authors also characterized absolutely minimum attaining $(\mathcal{AM})$ Hankel operators in the same article.

 In this article, we address the following question which is posed in \cite{RGSSSTOE1}.
 \begin{question}
 	Characterize all the $\mathcal{AM}$-Toeplitz operators and all the $\mathcal{AN}$-Hankel operators.
 \end{question}
 In \cite[Theorem 2.1]{RGSSSTOE1}, a characterization of an $\mathcal{AN}$-Toeplitz operator $T_\varphi, \varphi \in L^\infty$ is given as follows:
 \begin{theorem}
 		Let $\varphi \in L^\infty$. Then $T_\varphi \in \mathcal{AN}(H^2)$ if and only if $\|\varphi\|_\infty I - |T_\varphi| \in \mathcal{F}(H^2)_+$.
 \end{theorem}
Here we characterize such an operator in terms of the symbol $\varphi$ itself, which improves the about result.

Next we characterize $\mathcal{AM}$-Toeplitz operators and further prove that there are no non-compact $\mathcal{AN}$-Hankel operators. All these results are found in section 2.

In the remaining part of this section, we give a few notations and definitions which are necessary for developing the article.

\subsection{Preliminaries}
Let $H_1,H_2$ be infinite dimensional complex Hilbert spaces and $\mathcal{B}(H_1,H_2)$ denote the Banach space of all bounded linear operators from $H_1$ to $H_2$.
If $T \in \mathcal{B}(H_1,H_2)$, then $T$ is said to be finite rank if its range space is finite dimensional and  compact if $T$ maps any bounded set in $H_1$ to a pre-compact set in $H_2$.

\begin{definition}\cite[Definition 1.1, 1.2]{Carvjalnorm}
	An operator $T \in \mathcal{B}(H_1,H_2)$ is called norm attaining if there exists a unit vector $x \in H_1$ such that $\|Tx\| = \|T\|$. If for every closed subspace $M$ of $H_1$, $T|_M : M \to H_2$ is norm attaining, then $T$ is called an absolutely norm attaining or $\mathcal{AN}$-operator.
\end{definition}

\begin{definition}\cite[Definition 1.1, 1.4]{CARmin}
	An operator $T \in \mathcal{B}(H_1, H_2)$ is said to be minimum attaining if there exists a unit vector $x \in H_1$ such that $m(T) = \|Tx\|$, where \[m(T):= \inf\{\|Tx\| : x \in H_1, \|x\| = 1 \}.\]
	
	 If for every closed subspace $M$ of $H_1$, $T|_M : M \to H_2$ is minimum attaining, then $T$ is said to be absolutely minimum attaining or $\mathcal{AM}$-operator.
\end{definition}

For more details on $\mathcal{AN}$-operators we refer to \cite{Carvjalnorm,PandeyPaulsen,Rameshpara,venkuramesh} and for that on $\mathcal{AM}$-operators we refer to \cite{NeeruRamesh,CARmin,GAN}.

Let $\mathcal{F}(H^2)$ and $\mathcal{K}(H^2)$ denote the set of all finite rank and compact operators on $H^2$, respectively.
The set of all $\mathcal{AN}$ and $\mathcal{AM}$-operators on $H^2$ are denoted by $\mathcal{AN}(H^2)$ and $\mathcal{AM}(H^2)$, respectively.

Let $C(\mathbb{T})$ denote the set of all complex valued continuous functions on $\mathbb{T}$, $H^\infty = H^2 \cap L^\infty$ and the Sarason algebra $H^\infty + C(\mathbb{T})$ is a closed subalgebra of $L^\infty$. The notations and terminologies used in \cite{RGSSSTOE1} will be adhered throughout this article.

\section{Main Results}

We start this section by recalling a result which connects both Toeplitz and Hankel operators.

\begin{prop}\cite[Proposition 8]{YOS2} \label{H&Treln}
	For any $\varphi, \psi \in L^\infty$, we have $H^*_\psi H_\varphi = T_{\overline{\psi}\varphi} - T_{\overline{\psi}} T_\varphi$.
\end{prop}

The following results characterize all finite rank and compact Hankel operators.

\begin{theorem} \cite[Corollary 3.2, Page 21]{Peller} \label{finiteHankel}
Let $\varphi \in L^\infty$. Then $H_\varphi \in \mathcal{F}(H^2)$ if and only if $(I - P )\varphi$ is a rational function.
\end{theorem}

\begin{theorem}\cite[Corollary 4.3.3, Page 145]{MAR} \label{compactHankel} Let $\varphi \in L^\infty$. Then $H_\varphi \in \mathcal{K}(H^2)$ if and only if $\varphi \in H^\infty +C(\mathbb{T})$.
\end{theorem}

Now we characterize an $\mathcal{AN}$-Toeplitz operator $T_\varphi$ in terms of $\varphi$.

\begin{theorem}\label{ANToeplitz}
Let $\varphi \in L^\infty$. Then $T_\varphi \in \mathcal{AN}(H^2)$ if and only if $|\varphi| = \|\varphi\|_\infty$ a.e.$(\mu)$ and $(I-P)\varphi$ is a rational function.
\end{theorem}
\begin{proof}
	By \cite[Theorem 2.1 and Corollary 2.3]{RGSSSTOE1}, we have $\|\varphi\|^2_\infty I - T^*_\varphi T_\varphi = F$, where $F \in \mathcal{F}(H^2)$ and $|\varphi| = \|\varphi\|_\infty$ a.e.$(\mu)$. Hence by Proposition \ref{H&Treln}, we get $H^*_\varphi H_\varphi = F \in \mathcal{F}(H^2)$. Therefore by Theorem \ref{finiteHankel}, it follows that $(I-P)\varphi$ is a rational function.
	
Conversely, if $|\varphi| = \|\varphi\|_\infty$ a.e.$(\mu)$ and $(I-P)\varphi$ is a rational function, then $H^*_\varphi H_\varphi = F$, where $F \in \mathcal{F}(H^2)_{+}$. Again by Proposition \ref{H&Treln}, we get $\|\varphi\|^2_\infty I-T^*_\varphi T_\varphi =   F$. Hence by \cite[Theorem 2.1]{RGSSSTOE1}, $T_\varphi \in \mathcal{AN}(H^2)$.
\end{proof}
\begin{corollary}
	Let $\varphi \in L^\infty$. Then $T_\varphi \in \mathcal{AN}(H^2)$ if and only if $|\varphi| = \|\varphi\|_\infty$ a.e.$(\mu)$ and $H_\varphi \in \mathcal{AM}(H^2)$.
\end{corollary}
\begin{proof}
By \cite[Theorem 3.3]{RGSSSTOE1}, we have $H_\varphi \in \mathcal{AM}(H^2)$ if and only if $H_\varphi \in \mathcal{F}(H^2)$. Now, from Theorems \ref{finiteHankel} and \ref{ANToeplitz}, we get the required conclusion.
\end{proof}

\begin{remark} If $T_\varphi \in \mathcal{AN}(H^2)$, then it satisfies $T^*_\varphi T_\varphi + H^*_\varphi H_\varphi = \|\varphi\|^2_\infty I$.
\end{remark}

Next we characterize all $\mathcal{AM}$-Toeplitz operators.

\begin{theorem} \label{AMToeplitz}
	Let $\varphi \in L^\infty$. Then $T_\varphi \in \mathcal{AM}(H^2)$ if and only if $|\varphi|= \alpha$ a.e.$(\mu)$ and $\varphi \in H^\infty +C(\mathbb{T})$, where $\alpha \in \sigma_{ess}(|T_\varphi|)$.
\end{theorem}
\begin{proof}
	If $T_\varphi \in \mathcal{AM}(H^2)$,  then by \cite[Theorem 5.14]{GAN}, $T^*_\varphi T_\varphi \in \mathcal{AM}(H^2)$. From \cite[Theorem 5.8]{GAN}, we have $T_{\overline{\varphi}} T_\varphi = \alpha^2 I - K + F$, where $\alpha^2 \in \sigma_{ess}(T_{\overline{\varphi}}T_\varphi), \alpha \geq 0, K \in \mathcal{K}(H^2), F \in \mathcal{F}(H^2)$ with $K \leq \alpha^2 I$ and $KF=0$. Hence by \cite[Propsition 2.18]{RGSSSTOE1}, we get $|\varphi| = \alpha$ a.e.$(\mu)$ and by \cite[Lemma 3]{KOV}, $\alpha \in \sigma_{ess}(|T_\varphi|)$.
		By Proposition \ref{H&Treln}, we get $H^*_\varphi H_\varphi = K -F \in \mathcal{K}(H^2)$. Therefore $H_\varphi \in \mathcal{K}(H^2)$ and by Theorem \ref{compactHankel}, $\varphi \in H^\infty + C(\mathbb{T})$.
	
	Conversely, assume that $|\varphi| = \alpha$ a.e.$(\mu)$ and $\varphi \in H^\infty +C(\mathbb{T})$. Then $H^*_\varphi H_\varphi = K$, where $K \in \mathcal{K}(H^2)_{+}$. By Proposition \ref{H&Treln}, we get $T_{\overline{\varphi}}T_\varphi = \alpha^2I - K$. Then by \cite[Proposition 3.5]{RGSSSTOE1}, $T_{\overline{\varphi}}T_\varphi \in \mathcal{AM}(H^2)$ and hence $T_\varphi \in \mathcal{AM}(H^2)$ by \cite[Theorem 5.14]{GAN}.
\end{proof}

\begin{remark}
In \cite{RAMSSS}, it was proved that the operator norm closure of $\mathcal{AM}$-operators is the same as the operator norm closure of $\mathcal{AN}$-operators. Hence by Theorems \ref{ANToeplitz} and \ref{AMToeplitz}, we observe that the set of all $\mathcal{AN}$-Toeplitz operators is contained in the set of all $\mathcal{AM}$-Toeplitz operators. Further, the set  of all $\mathcal{AM}$-Toeplitz operators forms a closed subset of space of all Toeplitz operators on $H^2$.
\end{remark}

The next theorem shows that there is no non-compact absolutely norm attaining Hankel operator.  For a detailed study on compact Hankel operators we refer to \cite{Hartman}.
\begin{theorem} \label{ANHankel}
	Let $\varphi \in L^\infty$. Then $H_\varphi \in \mathcal{AN}(H^2)$ if and only if $H_\varphi \in \mathcal{K}(H^2)$.
\end{theorem}
\begin{proof}
	Clearly if $H_\varphi \in \mathcal{K}(H^2)$, then $H_\varphi \in \mathcal{AN}(H^2)$.
	
	On the other hand, if $H_\varphi \in \mathcal{AN}(H^2)$, then by \cite[Corollary 2.11]{venkuramesh}, $H^*_\varphi H_\varphi \in \mathcal{AN}(H^2)$. Hence by \cite[Theorem 2.5]{venkuramesh}, $H^*_\varphi H_\varphi = \alpha I - F + K$ for some $\alpha \geq 0$, $F \in \mathcal{F}(H^2)$, $K \in \mathcal{K}(H^2)$ with $F \leq \alpha I$ and $KF =0$. By Proposition \ref{H&Treln}, we get  $\alpha I - F + K = T_{|\varphi|^2} - T^*_\varphi T_\varphi$ or $T_{|\varphi|^2- \alpha} - T_{\overline{\varphi}} T_\varphi = K- F \in \mathcal{K}(H^2)$. So by \cite[Exercise 7.7]{DOU}, $|\varphi|^2- \alpha = \overline{\varphi}\varphi$, which implies $\alpha =0$. Hence $H^*_\varphi H_\varphi \in \mathcal{K}(H^2)$ which in turn results in $H_\varphi \in \mathcal{K}(H^2)$.
\end{proof}

\begin{corollary}
Let $\varphi \in L^\infty$. Then $H_\varphi \in \mathcal{AN}(H^2)$ if and only if $\varphi \in H^\infty + C(\mathbb{T})$.
\end{corollary}
\begin{proof}
The proof follows by Theorems \ref{compactHankel} and \ref{ANHankel}.
\end{proof}

\begin{remark}
	\begin{enumerate}
	\item The set of all $\mathcal{AM}$-Hankel operators is a  subset of the set of all $\mathcal{AN}$-Hankel operators.
		\item By Theorem \ref{ANHankel}, we have that the set of all $\mathcal{AN}$-Hankel operators is a closed linear subspace of space of all Hankel operators on $H^2$.
	\end{enumerate}
\end{remark}

We summarize the results of this paper.

For any $\varphi \in L^\infty$,
\begin{enumerate}
	\item $T_\varphi \in \mathcal{AN}(H^2)$ if and only if $|\varphi| = \|\varphi\|_\infty$ a.e.$(\mu)$ and $(I-P)\varphi$ is a rational function.
	\item $T_\varphi \in \mathcal{AM}(H^2)$ if and only if $|\varphi|= \alpha$ a.e.$(\mu)$ and $\varphi \in H^\infty +C(\mathbb{T})$, where $\alpha \in \sigma_{ess}(|T_\varphi|)$.
	\item $H_\varphi \in \mathcal{AN}(H^2)$ if and only if $\varphi \in H^\infty + C(\mathbb{T})$.
\end{enumerate}

	\section*{Funding}
The first author is supported by SERB Grant No. MTR/2019/001307, Govt. of India. The second author is supported by the  Department of Science and Technology- INSPIRE Fellowship (Grant No. DST/INSPIRE FELLOWSHIP/2018/IF180107).

\section*{Data Availability}
Data sharing is not applicable to this article as no datasets were generated or
analyzed during the current study.

\section*{Conflict of interest}
The authors declare that there are no conflicts of interests.

\end{document}